\documentclass[12pt,reqno]{amsart}
\usepackage{amsmath}
\usepackage{amssymb}
\newtheorem{thm}{Theorem}

\theoremstyle{definition}

\newcommand{\vertiii}[1]{{\left\vert\kern-0.25ex\left\vert\kern-0.25ex\left\vert
#1 \right\vert\kern-0.25ex\right\vert\kern-0.25ex\right\vert}}

\def \lim   {\text {\rm lim}}

\def \qed   {\hfill \vrule height6pt width 6pt depth 0pt}
\newcommand{\q}[1]{``#1''}

\begin{document}

\title[]{A Schur-Horn Theorem for symplectic eigenvalues}

\author[Rajendra Bhatia]{Rajendra Bhatia}
\address{Ashoka University, Sonepat\\ Haryana, 131029, India}

\email{rajendra.bhatia@ashoka.edu.in}

\author{Tanvi Jain}

\address{Indian Statistical Institute, New Delhi 110016, India}

\email{tanvi@isid.ac.in}

\begin{abstract}
Let $x$ and $y$ be positive $n$-vectors. We show that there exists a $2n\times 2n$ positive definite real matrix whose symplectic spectrum is $y,$ and the symplectic spectrum of whose diagonal is $x$ if and only if $x$ is weakly supermajorised by $y.$
\end{abstract}

\thanks{The work of RB is supported by a Bhatnagar Fellowship of the CSIR. 
TJ acknowledges financial support from {\sf SERB
MATRICS grant number MTR/2018/000554}.}



\maketitle

\section{Introduction}
Let $A$ be an $n\times n$ Hermitian matrix, and let $\Delta(A)$ and $\lambda(A)$ be the real $n$-vectors whose entries are the diagonal entries and the eigenvalues of $A,$ respectively. The celebrated Schur-Horn theorem gives a relationship between $\Delta(A)$ and $\lambda(A).$ This is described in terms of majorisation.\\\par
Let $x=(x_1,\ldots,x_n)$ be a vector in $\mathbb{R}^n.$ The co-ordinates of $x$ rearranged in decreasing order will be denoted by $x_{1}^{\downarrow} \geq \cdots \geq x_{n}^{\downarrow}$ and the same numbers listed in increasing order will be denoted as $x_{1}^{\uparrow} \leq \cdots \leq x_{n}^{\uparrow}.$ Let $x,y$ be two vectors in $\mathbb{R}^n.$ We say that $x$ is \emph{weakly submajorised} by $y,$ in symbols $x \prec_w y,$ if 
\begin{equation*}\label{eq1}
    \sum_{j=1}^{k} x_{j}^{\downarrow} \leq \sum_{j=1}^{k} y_{j}^{\downarrow}, \text{ for } 1\leq k \leq n.
\end{equation*}
In addition if, \begin{equation*}\label{eq2}
\sum_{j=1}^{n} x_{j}^{\downarrow} = \sum_{j=1}^{n} y_{j}^{\downarrow},
\end{equation*}
we say that $x$ is majorised by $y,$ in symbols $x\prec y.$ The book \cite{mo} provides an encyclopedic coverage of majorisation; a brief treatment is given in \cite{rbh}.\\\par
In 1923, I. Schur \cite{s} showed that for every Hermitian matrix $A,$ we have the majorisation $\Delta(A) \prec \lambda(A).$ In 1954, A. Horn \cite{h} proved the converse: if $x,y$ are two real $n$-vectors, with $x\prec y$ then there exists a real symmetric matrix A such that $x=\Delta(A)$ and $y=\lambda(A).$ See \cite{mo} Sect. 9.B.\\\par
One of the basic theorems on majorisation says that $x\prec y$ if and only if $x$ is in the convex polytope whose vertices are the vectors $y_{\sigma}$ obtained by permuting the coordinates of $y$ according to the permutation $\sigma.$ Using this, the Schur-Horn theorem may be reformulated as follows: Let $\lambda$ be a vector in $\mathbb{R}^n$ and $\Lambda$ the diagonal matrix with diagonal $\lambda.$ Let $U(n)$ be the group of unitary matrices and let $\mathcal{O}_{\lambda}=\{U\Lambda U^*:U \in U(n)\},$ the unitary orbit of $\Lambda.$ (This is the collection of all Hermitian matrices whose eigenvalues are $\lambda.$) Then the range of the map $\Delta: \mathcal{O}_{\lambda} \rightarrow \mathbb{R}^n$ coincides with the polytope whose vertices are the vectors $\lambda_\sigma.$ In this form, the theorem has been extended to the more general setting of Lie groups starting with B. Kostant \cite{k}, followed by \cite{a} and \cite{g}. \\\par
The goal of this paper is to formulate and prove a Schur-Horn like theorem for the action of the real symplectic group on real positive definite matrices. \\\par
Let $\mathbb{M}(2n)$ be the space of $2n\times 2n$ real matrices and let $\mathbb{P}(2n)$ be the subset consisting of real positive definite matrices. Let $J=\begin{bmatrix}
0 & I\\
-I & 0
\end{bmatrix}$
 and let $Sp(2n)=\{M\in\mathbb{M}(2n):M^TJM=J\}.$ This is the Lie group of real symplectic matrices. By a theorem of Williamson \cite{dms, w} for every real positive definite matrix $A,$ there exists a symplectic matrix $M$ such that 
 \begin{equation*}\label{eq3}
     M^TAM=\begin{bmatrix}
D & 0\\
0 & D
\end{bmatrix},
 \end{equation*}
 where $D$ is a positive diagonal matrix. The diagonal entries of $D$ are enumerated as $$d_1(A) \leq \cdots \leq d_n(A),$$ and are called the \emph{symplectic eigenvalues} of $A.$ They are uniquely determined by $A$ and are complete invariants for the orbits of $\mathbb{P}(2n)$ under the action of the group $Sp(2n).$ The matrix $A^{1/2}JA^{1/2}$ is skew-symmetric, and so its eigenvalues occur in conjugate imaginary pairs. The absolute values of these are the symplectic eigenvalues $d_j(A),$ $1\leq j\leq n.$ See \cite{bj}\.\\\par
 We denote the $n$-vector of symplectic eigenvalues of $A$ by $d_s(A).$ Let $\Delta(A)$ be the diagonal of $A$ and let $\Delta_s(A)=d_s(\Delta(A)).$ If the matrix $A$ is partitioned into $n\times n$ blocks as $A=\begin{bmatrix}
A_{11} & A_{12}\\
A_{21} & A_{22}
\end{bmatrix},$ then $\Delta_s(A)=[\Delta(A_{11})\Delta(A_{22})]^{1/2}.$
(Here both the product of two vectors and the square root are taken entrywise.)
This can be seen from the relationship between symplectic eigenvalues of $A$ and the eigenvalues of $JA.$ In our version of the Schur-Horn theorem we compare $\Delta_s(A)$ and $d_s(A).$ One major difference from the case of the classical Schur-Horn theorem is that while the unitary group $U(n)$ is compact, the group $Sp(2n)$ is not. \\\par
 Let $x,y \in \mathbb{R}^n.$
We say $x$ is \emph{weakly supermajorised} by $y,$ in symbols $x \prec^w y,$ if
 \begin{equation*}\label{eq4}
    \sum_{j=1}^{k} x_{j}^{\uparrow} \geq \sum_{j=1}^{k} y_{j}^{\uparrow}, \text{ for } 1\leq k \leq n.
\end{equation*}
It is easy to see that $x$ is majorised by $y$ if and only if it is both weakly submajorised and supermajorised by $y.$\\\\
Our symplectic version of the Schur-Horn theorem is the following:

\begin{thm}\label{thm1}
Let $A$ be any $2n\times 2n$ real positive definite matrix. Then we have the weak majorisation \begin{equation}\label{eq5}
\Delta_s(A)\prec^w d_s(A).
\end{equation}
Conversely, if $x,y$ are two $n$ vectors with positive entries such that $x\prec^w y,$ then there exists a $2n\times 2n$ real positive definite matrix $A$ such that $x=\Delta_s(A)$ and $y=d_s(A).$
\end{thm}

Just like the classical Schur-Horn theorem, this theorem can be reformulated as a convexity statement. The convex polytope of the classical theorem is replaced by an unbounded convex set. This is a manifestation of the noncompactness of the group $Sp(2n).$ From the theory of majorisation we know that $x\prec^w y$ if and only if there exists a vector $z$ such that $z\prec y$ and $z\leq x.$ See \cite{mo}, Ch.5, Sect. A.9.a. Thus 
\begin{equation}\label{eq6}
    \{x:x\prec^w y\} = \bigcup_{z\prec y}\{x:z\leq x\}.
\end{equation}
This is an unbounded convex set, and can be visualized as follows. Let $S$ be the convex polytope whose vertices are the permuted vectors $y_\sigma.$ At each point of $S$ attach a copy of the positive orthant $\mathbb{R}^{n}_{+}.$ The set in \eqref{eq6} is the union of all these. We denote this set by $\Sigma_y.$ Then, Theorem 1 can be restated as:

\begin{thm}\label{thm2}
Let $d=(d_1,\ldots,d_n)$ be a positive $n$-vector and $D$ the diagonal matrix with $d_1,\ldots,d_n$ on its diagonal. Let $$\mathcal{O}_d=\{M\begin{bmatrix}
D & 0\\
0 & D 
\end{bmatrix}M^T: M \in Sp(2n)\},$$ be the symplectic orbit of $\begin{bmatrix}
D & 0\\
0 & D 
\end{bmatrix}.$ (This is the collection of all real positive definite matrices with symplectic eigenvalues $d_1,\ldots,d_n.$) Then the range of the map $\Delta_s: \mathcal{O}_d \rightarrow \mathbb{R}^{n}$ is the convex set $\sum_d.$
\end{thm}

\section{Proof of the Theorem}
The weak majorisation \eqref{eq5} is a consequence of results proved in \cite{sanders} and in our paper \cite{bj}. Let $A=\begin{bmatrix}
A_{11} & A_{12}\\
A_{12}^{T} & A_{22} 
\end{bmatrix}$ and define the \emph{symplectic diagonal} of $A$ as $$\mathfrak{D}_s(A)=\begin{bmatrix}
\Delta(A_{11}) & \Delta(A_{12})\\
\Delta(A_{12}^{T}) & \Delta(A_{22}) 
\end{bmatrix}=\begin{bmatrix}
\Delta(A_{11}) & \Delta(A_{12})\\
\Delta(A_{12}) & \Delta(A_{22}) 
\end{bmatrix}.$$ In \cite{sanders} and \cite{bj} it has been shown that \begin{equation}\label{eq7}
    d_s(\mathfrak{D}_s(A))\prec^w d_s(A).
\end{equation} If $\Delta(A_{11})=(\alpha_1,\ldots,\alpha_n), \Delta(A_{12})=(\beta_1,\ldots,\beta_n)$ and $\Delta(A_{22})=(\gamma_1,\ldots,\gamma_n),$ then one can see that $d_s(\mathfrak{D}_s(A))$ is the $n$-vector with entries $(\alpha_j\gamma_j-\beta_j^2)^{1/2},$ whereas $\Delta_s(A)$ is the $n$-vector with entries $(\alpha_j\gamma_j)^{1/2}, 1\leq j\leq n.$ This shows that \begin{equation}
    d_s(\mathfrak{D}_s(A))\leq \Delta_s(A).
\end{equation} The relation \eqref{eq5} follows from \eqref{eq7} and (8).\\\par
Now let $x,y$ be two positive vectors with $x\prec^w y.$ We have to produce a $2n\times 2n$ real positive definite matrix $A$ such that $d_s(A)=y$ and $\Delta_s(A)=x.$ Because of \eqref{eq6} we can find a vector $z$ such that $z\prec y$ and $z\leq x.$ Denote by $Y$ the diagonal matrix with diagonal entries $y_1,\ldots,y_n.$ By Horn's Theorem, there exists a real orthogonal matrix $\Omega$ such that $\Delta(\Omega Y\Omega^T)=z.$ The matrix $\begin{bmatrix}
\Omega & 0\\
0 & \Omega 
\end{bmatrix}$ is symplectic. Let $$B=\begin{bmatrix}
\Omega & 0\\
0 & \Omega 
\end{bmatrix}\begin{bmatrix}
Y & 0\\
0 & Y 
\end{bmatrix}\begin{bmatrix}
\Omega^T & 0\\
0 & \Omega^T 
\end{bmatrix}$$
Then $B$ is positive definite and $d_s(B)=y.$ The diagonal $\Delta(B)=\begin{bmatrix}
Z & 0\\
0 & Z 
\end{bmatrix},$ and therefore $\Delta_s(B)=z.$ \\\par
Let $P,Q,R,S$ be $n\times n$ diagonal matrices with \begin{equation}
    PS-QR=I.
\end{equation} Then the matrix $M=\begin{bmatrix}
P & Q\\
R & S 
\end{bmatrix}$ is symplectic. Let $A=MBM^T.$ Then $$\Delta(A)=\begin{bmatrix}
PZP+QZQ & 0\\
0 & RZR+SZS 
\end{bmatrix}.$$The vector $\Delta_s(A)$ has entries
\begin{equation}
    [(p_j^2+q_j^2)(r_j^2+s_j^2)]^{1/2}z_j , \text{ }1\leq j \leq n.
\end{equation} 
We claim that we can choose $P,Q,R,S$ subject to the constraint (9) in such a way that the expressions in (10) assume all values $x_j\geq z_j, \ 1\leq j\leq n.$ For this, consider the group $SL(2,\mathbb{R})$ consisting of $2\times 2$ real matrices $\begin{bmatrix}
p & q\\
r & s
\end{bmatrix}$ with determinant 1. On this define the function $$f\left(\begin{bmatrix}
p & q\\
r & s
\end{bmatrix}\right)=[(p^2+q^2)(r^2+s^2)]^{1/2}.$$ This is a continuous function. Since $SL(2,\mathbb{R})$ is connected and unbounded, and $f(I)=1,$ the range of $f$ contains the interval $[1,\infty).$ This establishes our claim. \\\par
With $P,Q,R,S$ as above, and $M=\begin{bmatrix}
P & Q\\
R & S
\end{bmatrix},$ the matrix $A=MBM^T$ has the desired properties: $d_s(A)=y$ and $\Delta_s(A)=x.$ This completes the proof of Theorem \ref{thm1}.
\qed

\section{Remarks}
In spite of there being several differences between symplectic eigenvalues of positive definite matrices and ordinary eigenvalues, it is remarkable that, with appropriate interpretations, many known theorems for eigenvalues of Hermitian matrices have symplectic versions. See \cite{bj}, \cite{jm} and references therein. In the formulation of our Theorem 1 we have chosen one particular interpretation for the \q{diagonal} of $A.$ Another is the $n$-vector $d_s(\mathfrak{D}_s(A)).$ Given positive vectors $x,y,$ necessary and sufficient conditions for the existence of a positive definite matrix $A$ with $d_s(A)=y$ and $d_s(\mathfrak{D}_s(A))=x$ have been found in \cite{sanders}. These conditions consist of \eqref{eq7} and one additional inequality. Formally, this is an analogue of the Sing-Thompson theorem comparing the diagonal of a matrix and its singular values, which is quite different from the Schur-Horn theorem. \\\par
In Theorem \ref{thm1} we chose the geometric mean $[\Delta(A_{11})\Delta(A_{22})]^{1/2}$ of $\Delta(A_{11})$ and $\Delta(A_{22})$ for our \q{diagonal}.
We could have chosen the arithmetic mean
 $$\Delta_c(A)=\frac{\Delta(A_{11})+\Delta(A_{22})}{2},$$
and there is a good reason for this alternative choice.
In many calculations with symplectic eigenvalues expressions like $\frac{1}{2}[\langle u,Au\rangle+\langle v,Av\rangle],$ where $u,v$ are vectors with $\langle u,Jv\rangle =1,$ play the same role as the \q{Rayleigh quotients} $\langle x,Ax\rangle$ with $\langle x,x\rangle=1$ in the ordinary eigenvalue problem. With this choice again we get a version of the Schur-Horn theorem:
\begin{thm}\label{thm3}
Let $A$ be any $2n\times 2n$ positive definite matrix. Then we have the weak majorisation
\begin{equation}
    \Delta_c(A)\prec^w d_s(A).
\end{equation} Conversely, if $x,y$ are two positive $n$-vectors with $x\prec^w y,$ then there exists a $2n\times 2n$ positive definite matrix $A$ such that $x=\Delta_c(A)$ and $y=d_s(A).$
\end{thm}

\begin{proof}
Since $\Delta_s(A)\leq \Delta_c(A),$ the weak majorisation (11) is a consequence of \eqref{eq5}. To prove the converse, proceed as in the proof of Theorem 1. Define the matrix $B$ exactly in the same way as was done there. Let $\alpha=(\alpha_1,\ldots,\alpha_n)$ be an $n$-vector with positive entries and let $M_\alpha$ be the $2n\times 2n$ diagonal matrix with diagonal entries $(\alpha_1,\ldots,\alpha_n,\alpha_{1}^{-1},\ldots,\alpha_{n}^{-1}).$ This is a symplectic matrix. Let $A_\alpha=M_{\alpha} BM_{\alpha}^T.$ Then $d_s(A_{\alpha})=y$ for all $\alpha,$ and $$\Delta(A_{\alpha})=(\alpha_1z_1,\ldots,\alpha_nz_n,\alpha_{1}^{-1}z_1,\ldots,\alpha_{n}^{-1}z_n).$$ Hence, $$\Delta_c(A_{\alpha})=\frac{1}{2}((\alpha_1+\alpha_{1}^{-1})z_1,\ldots,(\alpha_n+\alpha_{n}^{-1})z_n).$$
Now $z_j\leq x_j$ for all $j.$ The function $f(t)=\frac{1}{2}(t+t^{-1})$ takes the value $1$ at $t=1$ and increases monotonically in $[1,\infty).$ Hence for each $j,$ there exists $\alpha_j\geq 1$ such that $x_j=\frac{1}{2}(\alpha_j+\alpha_{j}^{-1})z_j.$ For this $\alpha,$ let $A_\alpha=A.$ Then $A$ is a $2n\times 2n$ positive definite matrix with $d_s(A)=y$ and $\Delta_c(A)=x.$ This completes the proof.   
\end{proof}


\vskip.2in
\begin{thebibliography}{99}

\bibitem{dms}
Arvind, B. Dutta, N. Mukunda and R. Simon, {\it The real symplectic groups in quantum mechanics and optics}, Pramana, 45 (1995), 471-495.

  \bibitem{a}
	M. F. Atiyah, {\it Angular momentum, convex polyhedra and algebraic geometry}, Proc. Edinburgh Math. Soc., 26 (1983) 121-138.

\bibitem{rbh} R. Bhatia, {\it Matrix Analysis,} Springer, 1997.

\bibitem{bj} 
R. Bhatia, T. Jain, {\it On symplectic eigenvalues of positive definite matrices}, J.  Math. Phys., 56 (2015), 112201.

\bibitem{sanders} 
J. Eisert, T. Tyc, T. Rudolph, B.C. Sanders, {\it Gaussian quantum marginal problem}, Commun. Math. Phys., 280 (2008), 263-280.

  \bibitem{g}
	V. Guillemin and S. Sternberg, {\it Convexity properties of the moment
 mapping}, Invent. Math., 67 (1982) 491-513.
	
\bibitem{h}
A. Horn, {\it Doubly stochastic matrices and the diagonal of a rotation matrix}, Amer. J. Math., 76 (1954) 620-630.

\bibitem{jm}
T. Jain and H. K. Mishra, {\it Derivatives of symplectic eigenvalues and a Lidskii type theorem}, preprint.
  
	\bibitem{k}
	B. Kostant, {\it On convexity, the Weyl group and the Iwasawa decomposition}, Annales Scientifiques de L'École Normale Supérieure, 6 (1973) 413-455.

\bibitem{mo} A. W. Marshall, I. Olkin and B. C. Arnold, Inequalities: Theory of Majorization and Its Applications, Springer, 2011.	

\bibitem{s}
I. Schur, {\it Uber eine Klasse von
Mittelbildungen mit Anwendungen auf die Determinantentheorie},Sitzunsber.Berlin.Math.Ges., 22 (1923) 9-20.
  
	\bibitem{w} J. Williamson, {\it On the algebraic problem concerning the normal forms of linear dynamical systems}, Am. J. Math., 58 (1936) 141-163.

\end{thebibliography}
\end{document}